\documentclass[12pt,leqno,oneside]{amsart}
\usepackage{mathrsfs,dsfont}
\usepackage{amsmath,amstext,amsthm,amssymb,bbm}
\usepackage{charter}
\usepackage{typearea}
\usepackage{pdfsync}
\usepackage{hyperref}

\usepackage[a4paper,left=2cm, right=2cm,
top=2cm,bottom=2cm,footskip=1cm]{geometry}
\numberwithin{equation}{section}
  
\newcommand{\field}[1]{\mathbb{#1}}
  
\newcommand{\R}{\field{R}}   
\newcommand{\C}{\field{C}}      
\newcommand{\N}{\field{N}}  

\newcommand{\ddc}{\frac{i}{\pi}\,\partial\overline\partial}
  
\def\cC{\mathscr{C}}  
\def\cF{\mathscr{F}}      
      
\def\cL{\mathscr{L}}                       
\def\cO{\mathscr{O}}

\newcommand{\til}[1]{\widetilde{#1}}

\newcommand{\cali}[1]{\mathscr{#1}}
\newcommand{\cI}{\cali{I}}
\newcommand{\cJ}{\cali{J}}

\DeclareMathOperator{\Div}{Div}

\DeclareMathOperator{\rank}{rank}

\DeclareMathOperator{\supp}{supp}

\DeclareMathOperator{\vol}{vol}

\newcommand{\ddbar}{\overline\partial}
\newcommand{\dbar}{\partial}
\newcommand{\reg}{\mathrm{reg}}
\newcommand{\loc}{\mathrm{loc}}
\newtheorem{thm}{Theorem}[section]

\newtheorem{cor}[thm]{Corollary}
\theoremstyle{definition}
\newtheorem{rem}[thm]{Remark}
\theoremstyle{definition}
\newtheorem{defn}[thm]{Definition}

\newtheorem{example}[thm]{Example}

\newcommand{\be}{\begin{eqnarray}}
\newcommand{\ee}{\end{eqnarray}}

\newcommand{\wi}{\widetilde} 

\newcommand{\oh}{\widehat}
\newcommand{\comment}[1]{}

\begin{document}       
\title{Singular holomorphic Morse inequalities on non-compact manifolds}  
 
\author{Dan Coman}

\address{Dan Coman, Department of Mathematics, Syracuse University, Syracuse, NY 13244-1150, USA}
\email{dcoman@syr.edu}
\thanks{D.\ C.\ is partially supported
by the Simons Foundation Grant 853088 and the NSF Grant DMS-2154273}

\author{George Marinescu}
\address{Univerisit\"at zu K\"oln, Mathematisches institut,
Weyertal 86-90, 50931 K\"oln, Germany 
\newline\mbox{\quad}\,Institute of Mathematics `Simion Stoilow', 
Romanian Academy, Bucharest, Romania}
\email{gmarines@math.uni-koeln.de}
\thanks{G.\ M.\ is partially supported by 
 the DFG funded project SFB TRR 191 `Symplectic Structures in Geometry, 
 Algebra and Dynamics' (Project-ID 281071066\,--\,TRR 191) and the ANR-DFG project 
 QuaSiDy (Project-ID 490843120).}
 
\author{Huan Wang}  

\address{Huan Wang, Tong Jing Nan Lu 798, Suzhou, China}     
\email {huanwang2016@hotmail.com}  
\thanks{H.\ W.\ is partially supported by the Joint ICTP-IMU Research Fellowship}

\dedicatory{Dedicated to the memory of Mihnea Col\c{t}oiu}
\date{April 21, 2023}                 
\begin{abstract}    
We obtain asymptotic estimates of the dimension of cohomology 
on possibly non-compact complex manifolds for line bundles endowed 
with Hermitian metrics with algebraic singularities.
We give a unified approach to establishing singular holomorphic 
Morse inequalities for hyperconcave manifolds, pseudoconvex 
domains, $q$-convex manifolds and $q$-concave manifolds, 
and we generalize related estimates of Berndtsson. 
We also consider the case of metrics with more general
than algebraic singularities.
\end{abstract}          
\maketitle               
                                           
\tableofcontents                                 
   
\section{Introduction}\label{sec_intro}      
 
The aim of this article is to establish singular holomorphic Morse 
inequalities on complex manifolds satisfying certain convexity conditions.
The asymptotic estimates of the dimension of the cohomology groups
of high tensor powers of a holomorphic line bundle were motivated
by the Grauert-Riemenschneider conjecture \cite{GR:70}, 
which states that a compact complex manifold with a 
Hermitian holomorphic line bundle whose curvature form is 
positive definite on an open dense subset is a Moishezon manifold. 
Siu \cite{Si1:84} proved this conjecture using the 
Riemann-Roch-Hirzebruch formula and showed that for any 
$q>0$, $\dim H^q(X,L^p)=o(p^{n})$ as $p\rightarrow\infty$, 
for any compact complex manifold $X$ and semi-positive line bundle 
$L$. This can be seen 
as a refinement of the Kodaira-Serre vanishing theorem and of the 
Kodaira embedding theorem. 
Based on Siu's $\dbar\ddbar$ formula and the rescaling trick, 
Berndtsson \cite{Bern:02} improved Siu's result to the optimal 
size $p^{n-q}$, i.e. for every $q>0$ one has that 
$\dim H^q(X,L^p)=O(p^{n-q})$ as $p\rightarrow\infty$, 
for any compact complex manifold $X$ and semi-positive 
line bundle $L$. Later, Berndtsson's estimates 
were extended to various complex 
manifolds \cite{Wh:16,Wh:21b,Wh:21a} and line bundles with metrics 
with algebraic singularities \cite{Wh:21a}.

Motivated by Siu's solution and Witten's analytic proof of 
the standard Morse inequalities, Demailly \cite{De:85} 
established the holomorphic Morse inequalities, which strengthen 
Siu's solution by explicitly showing that 
\begin{equation}\label{e:Demailly}
\dim H^0(X,L^p)\geq \frac{p^n}{n!}\int_{X(\leq 1)}c_1(L,h^L)^n+o(p^n),\, p\rightarrow\infty,
\end{equation}
where $X$ is a compact complex manifold and $L$ is 
a line bundle with a smooth Hermitian metric $h^L$. 
A powerful tool for finding holomorphic sections, 
Demailly's holomorphic Morse inequalities have been extended 
to various classes of complex manifolds 
\cite{Bou:89,Mar92b,Mar96,Mar05,MarTC,MM07,TCM:01} and 
line bundles with metrics having algebraic singularities \cite{Bo:93}. 
It is noteworthy that Bonavero's singular holomorphic Morse inequalities 
\cite{Bo:93} on compact complex manifolds, as well as the criteria of 
Ji-Shiffman \cite{JS:93} and Takayama \cite{Ta:94}, 
provide a complete characterisation for Moishezon manifolds, and 
the bigness of line bundles \cite{MM07}. 
Together with Fujita's approximate Zariski decomposition \cite{Fuji:94} 
and Demailly's approximation theorem for positive closed currents \cite{Dem:92}, 
the singular holomorphic Morse inequalities lead to Boucksom's volume 
formula for pseudoeffective line bundles on compact K\"{a}hler 
manifolds \cite{Bouck:02}. 
We refer to \cite{Dem_analmeth_2012,MM07}
for comprehensive studies of the holomorphic Morse inequalities.
       
A natural problem is to establish holomorphic Morse inequalities, 
as well as Siu-Berndtsson type estimates, for complex manifolds 
possessing a line bundle endowed
with a metric with algebraic singularities. 
In this paper we consider the case when the singular locus of such a metric is compact. 
 
Let $M$ be a connected
complex manifold of dimension $n$, $L$ be a holomorphic 
line bundle on $M$ and $h^L$ be a Hermitian metric on $L$ 
with algebraic singularities, see Section \ref{sec_algsing}. 
We denote by $\cI(h^L)$  
the Nadel multiplier ideal sheaf of the
Hermitian metric $h^L$, cf.\ Definition \ref{D:NMI}.
Let $R^{(L,h^L)}$ be the curvature current of $(L,h^L)$ 
and set $c_1(L,h^L)=\frac{i}{2\pi}\,R^{(L,h^L)}$. 
Let $S(h^L)$ be the singular locus of $h^L$, 
which is a closed analytic subset in $M$. 
On the set $M\setminus S(h^L)$ the metric $h^L$
and its curvature $R^{(L,h^L)}$ are smooth.
For $q\in\{0,1,\ldots,n\}$ and $x\in M\setminus S(h^L)$ 
we say that the curvature $R^{(L,h^L)}_x$ has signature
$(q,n-q)$ if it has $q$ negative eigenvalues and 
$n-q$ positive eigenvalues as a Hermitian endomorphism of
$T^{(1,0)}_xX$. 
We introduce the $q$-index set 
\begin{equation}\label{eq:mq}
M(q)=M(q,h^L)=\{x\in M\setminus S(h^L):\text{$R^{(L,h^L)}_x$ 
has signature $(q,n-q)$}\},
\end{equation}
and we set 
\[M(\geq r):=\bigcup_{q=r}^n M(q)\,,\,\;M(\leq r):=
\bigcup_{q=0}^r M(q).\] 
It is clear that $M(q)$, $M(\geq r)$ and $M(\leq r)$ 
are open subsets of $M$. 
  
The first result deals with singular holomorphic Morse 
inequalities for hyperconcave manifolds. This subclass
of the class of $1$-concave manifolds
was introduced and studied by Mihnea Col\c{t}oiu \cite{Colt:90,Colt:91},
see also \cite{MD06}. 
\begin{thm}\label{thm_hyper}
Let $M$ be a hyperconcave manifold of dimension $n$ and 
$L$ be a holomorphic line bundle on $M$. Let $h^L$ be a 
Hermitian metric on $L$ with algebraic singularities such 
that $S(h^L)\subset Z$ and $c_1(L,h^L)\geq 0$ on $M\setminus Z$ 
for some compact $Z\subset M$. Then, as $p\to\infty$,
\[\dim H^0(M,L^p\otimes K_M\otimes \cI(h^{L^p}))\geq 
\frac{p^n}{n!}\int_{M(\leq 1)}c_1(L,h^L)^n+o(p^n).\]
\end{thm}       
If the metric $h^L$ is smooth,
Theorem \ref{thm_hyper} 
reduces to \cite[Theorem 3.4.9]{MM07}. 
As consequences of Theorem \ref{thm_hyper}, 
we obtain an estimate for the adjoint volume of a line bundle 
(Corollary \ref{cor_ajvol}), and a Siu-Demailly type criterion 
for Moishezon spaces with isolated singularities (Corollary \ref{cor_moisp}).  
Theorem \ref{thm_hyper} implies a version of Bonavero's singular 
holomorphic Morse inequality for certain metrics with more 
general singularities than the algebraic ones, cf.\ Theorem \ref{T:hyperapp}.
 
We consider next the case of singular holomorphic Morse 
inequalities on pseudoconvex domains. 
 
 \begin{thm}\label{thm_pcvx}
 	Let $M\Subset V$ be a smooth pseudoconvex 
	domain in a complex manifold $V$ of dimension 
	$n$ and $L,E$ be holomorphic vector bundles on $V$ 
	with $\rank (L)=1$. Let $h^L$ be a Hermitian metric on 
	$L$ with algebraic singularities such that $S(h^L)\subset M$ 
	and $c_1(L,h^L)>0$ on the boundary of $M$. 
	Then, as $p\rightarrow \infty$, 
 	\begin{equation}\label{eq_pdchmi}
 	\dim H^0(M,L^p\otimes E\otimes \cI(h^{L^p}))\geq
	\rank(E)\,\frac{p^n}{n!}\int_{M(\leq 1)} c_1(L,h^L)^n+o(p^n).
 	\end{equation}
\end{thm} 
  
Our third result deals with singular holomorphic Morse inequalities 
on $q$-convex manifolds.   
       
\begin{thm}\label{thm_hmi_q}
	Let $q,s\in \N$, $1\leq q,s\leq n$, $M$ be a $q$-convex 
	manifold of dimension $n$, and $L,E$ be holomorphic 
	vector bundles on $M$ with $\rank(L)=1$. Let $h^L$ 
	be a Hermitian metric on $L$ with algebraic singularities 
	such that $S(h^L)\subset Z$ and $c_1(L,h^L)$ has at least 
	$n-s+1$ non-negative eigenvalues on $M\setminus Z$, 
	for some compact set $Z\subset M$. Then, for any $\ell\geq s+q-1$, 
	the following strong and weak Morse inequalities hold as $p\to\infty$:
\begin{equation} 
\sum_{j=\ell}^n(-1)^{\ell-j}
\dim H^j(M, L^p\otimes  E\otimes\cI(h^{ L^p}))
\leq \rank(E)\,\frac{p^n}{n!}
\int_{M(\geq\ell)}(-1)^\ell c_1(L,h^L)^n+o(p^n), \label{eq_hmi_q}
\end{equation}
\begin{equation} 
\dim H^\ell(M, L^p\otimes  E\otimes \cI(h^{ L^p}))\leq 
\rank(E)\,\frac{p^n}{n!}\int_{M(\ell)}(-1)^\ell c_1(L,h^L)^n+o(p^n). 
\label{eq_hmi_qw}
\end{equation}
\end{thm} 
Note that the hypotheses of Theorem \ref{thm_hmi_q} imply
that for $\ell\geq s+q-1$ we have $M(\ell)=Z(\ell)$, thus
the integrals on the right-hand side of the Morse inequalities are
finite.

Theorem \ref{thm_hmi_q} is a generalization for singular 
Hermitian metrics of the main theorem of \cite{Bou:89}.
If $c_1(L,h)\geq0$ on $X\setminus S(h^L)$ then by \eqref{eq_hmi_qw}, 
$\dim H^j(M, L^p\otimes  E\otimes \cI(h^{ L^p}))=o(p^n)$ for $j\geq q$.  
This can be improved as follows.

\begin{thm}\label{thm_qcvx_bb}
	Let $M$ be a $q$-convex manifold of dimension $n$, $1\leq q\leq n$, let $K$ be the exceptional set of $M$, and $L,E$ be holomorphic  line bundles on $M$. Let $h^L$ be a Hermitian metric on $L$ with algebraic singularities such that $S(h^L)$ is compact and $c_1(L,h^L)\geq 0$ on $U$ (in the sense of currents), where $U\subset M$ is open and $K\cup S(h^L)\subset U$. Then there exists $C>0$ such that for every $j\geq q$ and $p\geq 1$,
	\begin{equation}
		\dim H^{j}(M,L^p\otimes E\otimes \cI(h^{L^p}))\leq Cp^{n-j}. 
	\end{equation}    
\end{thm}     

When $M$ is compact (hence $1$-convex), Theorem \ref{thm_qcvx_bb} 
reduces to \cite[Theorem 1.7]{Wh:21a}, and when 
$S(h^L)=\emptyset$ to \cite[Theorem 1.5]{Wh:21b}. 
When $M$ is compact and $S(h^L)=\emptyset$ this
is of course Berndtsson's result \cite{Bern:02}.

The paper is organized as follows. 
In Section \ref{S:sec2} we introduce the notations and 
recall how to reduce the case of metrics with algebraic 
singularities to that of smooth metrics, following \cite{Bo:93}. 
The main results are proved in Section \ref{S:hmi}.

\medskip
\textit{This paper is dedicated to the memory of Mihnea Col\c{t}oiu, for
his many fundamental contributions to the convexity theory of
complex spaces, brilliant solutions to difficult open problems and his
inspiring mathematical personality.
He will be fondly remembered.}

\section{Reduction to the case of smooth metrics}\label{S:sec2}    

The proof of the singular Morse inequalities follows the methods of 
Bonavero \cite{Bo:93} (see also \cite[Section 2.3.2]{MM07}). 
That is, one uses a proper modification such that the pull-back 
of the curvature current of $h^L$ has singularities along 
a divisor with normal crossings. Then one introduces a modified 
Hermitian holomorphic line bundle with cohomology groups isomorphic to the 
original ones. Since the singular locus of $h^L$ is compact, 
the convexity of the ambient manifold is preserved 
by the proper modification, which allows us to reduce 
the singular case to the smooth one in Section \ref{S:hmi}. 
We recall this construction in Theorem \ref{thm_tech} and give 
an outline of its proof. We begin with a brief discussion of singular Hermitian metrics.

\subsection{Singular metrics with algebraic singularities}\label{sec_algsing} 

Let $M$ be a connected complex manifold of dimension $n$ and 
$\cO_M$ denote its structure sheaf. 
Let $\omega$ be a Hermitian form on $M$ and set $dv_{M}=\omega^n/n!$. 
We denote by $H^q(M,\cF)$, where $0\leq q\leq n$, 
the $q$-th cohomology group of a sheaf $\cF$ on $M$. 
If $F$ is a holomorphic vector bundle on $M$ and $\cO_M(F)$ 
is the sheaf of holomorphic sections of $F$ we set $H^q(M,F):=H^q(M,\cO_M(F))$. 

A function $\varphi:M\rightarrow [-\infty,+\infty)$ which 
is locally the sum of a plurisubharmonic (psh) function and a smooth function
is called quasi-plurisubharmonic (quasipsh). 

Let $L$ be a holomorphic line bundle on $M$ and $h^L_0$ be a 
smooth Hermitian metric on $L$. If $h^L$ is a singular Hermitian metric 
on $L$ (cf.\ \cite{Dem:92,MM07}) then $h^L=h^L_0e^{-2\varphi}$ for some real function 
$\varphi\in L^1_{\loc}(M)$. 
The curvature currents of $(L,h^L)$ are defined by
\[R^{(L,h^L)}=R^{(L,h^L_0)}+
2\dbar\ddbar\varphi,\quad c_1(L,h^L)=\frac{i}{2\pi}\,R^{(L,h^L)}=
c_1(L,h^L_0)+\ddc\varphi.\]
We denote by $R(h^L)$ the largest open subset of $M$ 
where $h^L$ (or equivalently $\varphi$) is smooth, 
and we call $S(h^L):=M\setminus R(h^L)$ the singular locus of $h^L$. 


We introduce the following important class of
singular Hermitian metrics, cf.\ \cite{Bo:93, Bouck:02,Dem_analmeth_2012} 

\begin{defn}\label{D:algsing}
A function $\varphi$ on $M$ is said to
have \emph{analytic singularities} if there exists a coherent ideal sheaf $\cI\subset\cO_M$
and $c>0$ such that $\varphi$ can be written locally as 
\begin{equation}\label{e:algsing}
\varphi=\frac{c}{2}\,\log\big(\sum_{j=1}^m|f_j|^2\big)+\psi,
\end{equation}
where $f_1,\ldots,f_m$ are local generators of the ideal sheaf $\cI$ and $\psi$ is a smooth function. 
If $c$ is rational, we furthermore say that $\varphi$ has
\emph{algebraic singularities}. Note that a function
with analytic singularities is quasipsh, and that its singular locus
is the support of the subscheme $V(\cI)$ defined by $\cI$. If $L$
is a holomorphic line bundle on $X$ and $h^L$ is
a singular Hermitian metric on $L$, written $h^L=h^L_0e^{-2\varphi}$
where $h^L_0$ is smooth,
we say that $h^L$ has analytic (resp.\ algebraic) singularities if
$\varphi$ has analytic (resp.\ algebraic) singularities.
\end{defn}

The following notion was introduced by Nadel \cite{Nad:90}, see also
\cite{Dem:92,Dem_analmeth_2012}.
\begin{defn}\label{D:NMI}
The Nadel multiplier ideal sheaf $\cI(\varphi)$ of a real locally
integrable function $\varphi$ on $M$
is the sheaf of germs of holomorphic functions $f$ such
that $|f|^2e^{-2\varphi}$ is locally integrable.
We denote by $\cI(h^L):=\cI(\varphi)$  
the Nadel multiplier ideal sheaf of $h^L=h^L_0e^{-2\varphi}$. 
\end{defn}

Clearly, the ideal sheaf $\cI(h^L)$ is independent 
on the choice of the Hermitian metric $h^L_0$. Since the 
Nadel multiplier ideal sheaf $\cJ(\varphi)$ of a psh (thus also quasipsh)
function $\varphi$ is coherent \cite{Nad:90} (cf.\ also 
\cite[(5.7) Proposition]{Dem:92,Dem_analmeth_2012}), 
it follows that $\cJ(h^L)$ is a coherent analytic sheaf on $M$
for any Hermitian metric $h^L$ 
with analytic singularities.

\subsection{Resolving algebraic singularities}\label{sec_teck}     

Let $M$ be a connected complex manifold of dimension $n$. 
We need the following theorem about the resolution of singularities 
(see \cite[Theorems 3, 4, 5.4.2]{AHV:18} and 
\cite[Theorem 1.10, 13.4]{BieMil:97}). 

\begin{thm}\label{thm_key}
Let $\cF$ be a coherent ideal sheaf on $M$ such that 
\[Y:=\supp(\cO_M/\cF)=\{ x\in M:\cF_x\neq\cO_{M,x}\}\] 
is compact. Then there exits a complex manifold $\widetilde{M}$ 
and a proper modification $\pi: \til M\rightarrow M$, 
given as the composition of finitely many blow-ups with smooth center, such that:	 

(i) the restriction $\pi:\til M\setminus \pi^{-1}(Y)\rightarrow M\setminus Y$ is biholomorphic;

(ii) the pullback $\til \cF=\pi^{-1}\cF$ is locally normal crossings everywhere in $\til M$, i.e., for every point $x\in \til M$ there exits a coordinate neighborhood $W$ centered at $x$ and a monomial $h\in \cO_{\til M}(W)$ such that $\til\cF(W)$ is the principal ideal generated by $h$.
\end{thm}   

The holomorphic Morse inequalities for a singular metric are 
obtained from the corresponding inequalities for a suitable smooth 
metric by the following theorem. We denote by $ \lfloor a\rfloor$ 
the integer part of $a\in \R$. 
  
 \begin{thm}\label{thm_tech}
Let $(L,h^L)$ be a holomorphic line bundle on $M$ such that $h^L$ has algebraic singularities as in \eqref{e:algsing} and $S(h^L)$ is compact. There exists a proper modification $\pi:\til M\rightarrow M$, given by a composition of finitely many blow-ups with smooth center, such that $\pi:\til M\setminus \pi^{-1}(S(h^L))\rightarrow M\setminus S(h^L)$ is biholomorphic and the following hold: 
 		
(A) The weight $\til \varphi=\varphi\circ\pi$ of the metric $h^{\til L}= \pi^* h^L=(\pi^* h^L_0)e^{-2\til\varphi}$ on $\til L= \pi^* L$ has the form 
 		\begin{equation}\label{e:tilph}
 		\til \varphi=c\sum_{j=1}^k c_j\log |g_j|+\til \psi 
 		\end{equation}
in local holomorphic coordinates at any given point $\til x$ in $\til M$, where $\til \psi$ is a smooth function, $c_j\in\N\setminus \{0\}$, $g_j$ are irreducible in $\cO_{\til M,\til x}\,$, and they define a global divisor $\sum_{j=1}^k c_j\til D_j$ that has only normal crossings and support $\pi^{-1}(S(h^L))$. Moreover, 
\begin{equation}\label{e:nlp}
\cI(h^{\til L^p})=\cI(p\til \varphi)=
\cO_{\til M}\big(-\sum_{j=1}^k \lfloor\, pcc_j\rfloor\,\til D_j\big), \text{ for all } p\geq1.
\end{equation}

(B) Let $c=r/m$, where $r,m$ are positive integers, and set 
\begin{equation}\label{e:dvd}
\til D:=r \sum_{j=1}^k c_j \til D_j\,,\,\; \oh L:=\til L^m\otimes\cO_{\til M}(-\til D)=\til L^m\otimes \cI(h^{\til L^m}).
\end{equation}
Then there exists a smooth Hermitian metric $h^{\widehat L}$ on $\widehat L$ such that 
\begin{equation}\label{e:c1}
c_1{(\oh L, h^{\oh L})}=mc_1{(\til L, h^{\til L})} \text{ on } \til M\setminus \til D.
\end{equation}

(C) If $E$ is a holomorphic vector bundle on $M$ then, for $p$ sufficiently large, we have 
\begin{align}
H^j(M,  L^p\otimes E\otimes K_M\otimes \cI(h^{L^p}))
&\cong H^j(\til M,\til L^p\otimes \til E\otimes K_{\til M}\otimes\cI(h^{\til L^p})),\label{e:isom}\\
H^j(M,  L^p\otimes E\otimes \cI(h^{L^p}))
&\cong H^j(\til M,\til L^p\otimes \til E\otimes \til{K_M^*}\otimes K_{\til M}\otimes \cI(h^{\til L^p})),\label{e:isoms}
\end{align}
for $0\leq j\leq n$, where $\til E=\pi^* E$ and $\til{K_M^*}=\pi^*(K_M^*)$.
\end{thm} 

\begin{proof}
{\em (A)} 
Let us apply the Theorem \ref{thm_key} for the coherent ideal
sheaf $\cI$ from the Definition \ref{D:algsing}.
Let $\pi: \til M\rightarrow M$ be as in Theorem \ref{thm_key}
and let $g$ be the local generator of the ideal $\pi^{-1}\cI$
in a neighborhood of a point $\widetilde{x}\in\widetilde{M}$.
Let $\{f_j:j=1,\ldots,m\}$ be generators of $\cI$ near 
$x=\pi(\widetilde{x})$.
Since $\{f_j\circ\pi:j=1,\ldots,m\}$ are generators of $\pi^{-1}\cF$
near $\widetilde{x}$ there exists holomorphic functions $h_j$
such that $f_j\circ \pi=gh_j$, where $h_j$ have no common zeros. 
It follows that 
\[ \til \varphi=\frac{c}{2}\log\big(\sum_{j=1}^m|f_j\circ  \pi|^2\big)
+\psi\circ\pi=c\log|g|+\til\psi,\]
where $\til\psi$ is smooth. We write $g=\prod_{j=1}^k g_j^{c_j}$ 
with $g_j$ irreducible factors, and consider the global divisors $\til D_j$ 
defined locally by $g_j$. Hence $\sum_{j=1}^k c_j\til D_j$ 
is a divisor with only normal crossings. 
Since $\til\psi$ is smooth, this implies \eqref{e:nlp}.
\smallskip

\noindent{\em (B)} Let $s_{\til D}$ be the canonical section of $\cO_{\til M}(-\til D)$ such that $\Div(s_{\til D})=-\til D$. We endow $\cO_{\til M}(-\til D)$ with a singular metric $h^{\til D}$ such that $|s_{\til D}|_{h^{\til D}}=1$ on $\til M\setminus \til D$, and we consider the metric  $h^{\oh L}=h^{\til L^m}\otimes h^{\til D}$ on $\widehat L$. Since the local weight of $h^{\til D}$ is $-r\sum_{j=1}^k c_j\log|g_j|$, we infer from \eqref{e:tilph} that $h^{\oh L}$ is smooth and \eqref{e:c1} holds.

\smallskip

\noindent{\em (C)} This follows by the same local arguments 
as in \cite{Bo:93} (see also \cite[pp.\ 106-109, (2.3.45)]{MM07}),
by using the Leray theorem about the cohomology of a direct image
of a sheaf and the Nadel vanishing theorem for
weakly 1-complete manifolds. 
We emphasize that it is essential here that the proper modification 
$\pi$ is the composition of finitely many blow-ups with smooth center. 
Note that \eqref{e:isoms} follows at once from \eqref{e:isom}.
\end{proof}

\section{Singular Morse inequalities}\label{S:hmi}   

This section is devoted to the proofs of our main results.
We state first the singular holomorphic Morse inequalities
for compact manifolds due to Bonavero \cite{Bo:93}.

\begin{thm}
\label{C:Bona}
Let $M$ be a compact complex manifold of dimension $n$, 
and $L,E$ be holomorphic vector bundles on $X$ with $\rank(L)=1$. 
Let $h^L$ be a Hermitian metric on $L$ with algebraic singularities. 
Then, for $0\leq q\leq n$, we have as $p\rightarrow \infty$ that
\begin{gather}
\dim H^q(M,L^p\otimes E\otimes\cI(h^{L^p}))\leq
\rank(E)\,\frac{p^n}{n!}\int_{M(q)}(-1)^q c_1(L,h^L)^n+o(p^n),
\label{eq:wm1}\\
\sum_{j=0}^q(-1)^{q-j}\dim H^j(M,L^p\otimes E\otimes\cI(h^{L^p}))\leq 
\rank(E)\,\frac{p^n}{n!}\int_{M(\leq q)}(-1)^q c_1(L,h^L)^n+o(p^n),
\label{eq:sm}	
\end{gather}
with equality for $q=n$ (the asymptotic Riemann-Roch-Hirzebruch 
formula for singular metrics). 
\end{thm}   
As a consequence,
\begin{equation}\label{eq:wm2}
\dim H^q(M, L^p\otimes  E\otimes \cI(h^{ L^p}))\geq 
\rank(E)\,\frac{p^n}{n!}
\int_{M(q-1)\cup M(q)\cup M(q+1)}(-1)^q c_1(L,h^L)^n+o(p^n).
\end{equation}
Moreover, for if for some $0\leq q\leq n$ we have 
$M(q-1)=M(q+1)=\emptyset$, then
\begin{equation}\label{eq:wm23}
\lim_{p\rightarrow \infty}n!\,p^{-n}
\dim H^q(M,L^p\otimes  E\otimes \cI(h^{ L^p}))=
\rank(E)\int_{M(q)}(-1)^q c_1(L,h^L)^n.
\end{equation}
By the notation $\int_{M(q)}c_1(L,h^L)^n$ we also assume
that the set $M(q)$ refers to the metric $h^L$, that is,
$M(q)=M(q,h^L)$ (cf.\ \eqref{eq:mq}). 
Note that the integrals on the right-hand side 
of the Morse inequalities are finite. By \eqref{e:c1} we have
$\pi^{-1}(M(\ell))=\til M(\ell)\setminus\til D$
and the integral of $c_1(L,h^L)^n$ on $M(\ell)$
equals the integral of the everywhere smooth form $m^{-n}c_1(\oh L,h^{\oh L})^n$
on $\til M(\ell)\setminus\til D$.

In the following we consider manifolds $M$ satisfying various
convexity conditions (such as $q$-convexity, weakly $1$-completeness, $q$-concavity,
hyperconcavity). For a coherent analytic sheaf $\cF$ on such manifold $M$
the cohomology spaces
$H^j(M,\cF)$ are finite dimensional only for some values of $j\in\{0,1,\ldots,n\}$
and the Morse inequalities hold for these values, but also for
some connected values (for example $j=0$).

\subsection{$q$-convex manifolds}\label{SS:qc}
According to \cite{AG:62}, a complex manifold $M$ of dimension $n$ 
is called \emph{$q$-convex} for some $q\in\{1,\ldots,n\}$
if there exists a smooth function $\varphi:M\longrightarrow [a,b)$, 
where $a\in\R$, $b\in\R\cup\{+\infty\}$, 
such that $M_c=\{\varphi<c\}\Subset M$ for all $c\in[a,b)$ and 
$i\partial\ddbar \varphi$ has at least $n-q+1$ positive eigenvalues 
on $M\setminus K$
for a compact subset $K\subset M$. We call $K$ the {\em exceptional set} of $M$. 
By the Andreotti-Grauert theory \cite{AG:62}, $H^j(X,\cF)$ is finite dimensional
for any $j\geq q$ and any coherent analytic sheaf $\cF$ on a $q$-convex manifold $X$.


The smooth version of the holomorphic Morse inequalities for $q$-convex 
manifolds is the following, see \cite[Theorem 0.1]{Bou:89},
\cite[Theorem 3.5.8]{MM07}.

\begin{thm}\label{lem_qhmi} 
Let $q,s\in \N$, $1\leq q,s\leq n$, $M$ be a 
$q$-convex manifold of dimension $n$, and $L,E$ 
be holomorphic vector bundles on $M$ with $\rank(L)=1$. 
Let $h^L$ be a Hermitian metric on $L$ such that $c_1(L,h^L)$ 
has at least $n-s+1$ non-negative eigenvalues on $M\setminus Z$ 
for some compact $Z\subset M$. Then for any $\ell\geq s+q-1$, 
the following strong Morse inequality holds as $p\rightarrow \infty$,
\[\sum_{j=\ell}^n(-1)^{j-\ell}\dim H^j(M, L^p\otimes  E)\leq 
\rank(E)\,\frac{p^n}{n!}\int_{M(\geq\ell)}(-1)^\ell c_1(L,h^L)^n+o(p^n).\]
\end{thm} 	
\begin{proof}[Proof of Theorem \ref{thm_hmi_q}]
Let $\pi:\til M\rightarrow M$ be the proper modification provided by Theorem \ref{thm_tech}. Then  $\til D=\pi^{-1}(S(h^L))\subset\pi^{-1}(Z)$ and $\pi:\til M\setminus\til D\to M\setminus S(h^L)$ is biholomorphic. So $\til M$ is also $q$-convex with exceptional set $\til D\cup\pi^{-1}(K)$, where $K$ is the exceptional set of $M$. Moreover, by \eqref{e:c1}, $c_1(\oh L,h^{\oh L})=mc_1(\til L,h^{\til L})$ has at least $n-s+1$ non-negative eigenvalues on $\til M\setminus\pi^{-1}(Z)$.

We write $p=mp'+m'$, where $p',m'\in\N$, $0\leq m'<m$. We infer by \eqref{e:nlp} and \eqref{e:dvd} that 
\begin{equation}\label{e:decomp}
\til L^p\otimes\cI(h^{\til L^p})=\oh L^{p'}\otimes\til L^{m'}\otimes\cO_{\til M}\big(-\sum_{j=1}^k \lfloor\, m'cc_j \rfloor\,\til D_j\big).
\end{equation}
Let $\til E=\pi^* E$, $\til{K_M^*}=\pi^*(K_M^*)$, and
\begin{equation}\label{e:remainder}
F_{m'}:=\til E\otimes\til{K_M^*}\otimes K_{\til M}\otimes\til L^{m'}\otimes\cO_{\til M}\big(-\sum_{j=1}^k \lfloor\, m'cc_j \rfloor\,\til D_j\big)\,,\,\;0\leq m'<m.
\end{equation}	
Then $\rank(F_{m'})=\rank(E)$.	By \eqref{e:isoms} we have, for $p=mp'+m'$ sufficiently large and for each $0\leq j\leq n$, that
\begin{equation}\label{e:cohomology}
H^j(M,  L^p\otimes E\otimes \cI(h^{L^p}))\cong H^j(\til M,\oh L^{p'}\otimes F_{m'}).
\end{equation}
Applying Theorem \ref{lem_qhmi} on $\til M$ 
to the Hermitian holomorphic line bundle $(\oh L,h^{\oh L})$ 
and to each $F_{m'}$, $0\leq m'<m$, we get for $\ell\geq s+q-1$ 
and all $p$ sufficiently large that 
\begin{align*}
\sum_{j=\ell}^n(-1)^{j-\ell}\dim H^j(M,L^p\otimes E\otimes\cI(h^{L^p}))&=
\sum_{j=\ell}^n(-1)^{j-\ell}\dim H^j(\til M,\oh L^{p'}\otimes F_{m'}) \\
&\leq\rank(E)\,\frac{(p')^n}{n!}\int_{\til M(\geq\ell)}(-1)^\ell 
c_1(\oh L,h^{\oh L})^n+o(p^n).
\end{align*}	
Since $c_1(\oh L,h^{\oh L})$ has at least $n-s+1$ non-negative 
eigenvalues on $\til M\setminus\pi^{-1}(Z)$ it follows that 
$\til M(\geq\ell)\subset\pi^{-1}(Z)$. The latter set is compact, 
so the above integral exists. 
Note that $M(\geq\ell)\subset M\setminus S(h^L)$. 
Using \eqref{e:c1} we infer that 
$\til M(\geq\ell)\setminus\til D=\pi^{-1}(M(\geq\ell))$ and 
\begin{equation}
\int_{\til M(\geq\ell)} c_1(\oh L,h^{\oh L})^n=
m^n\int_{\til M(\geq\ell)\setminus\til D} c_1(\til L,h^{\til L})^n=
m^n\int_{M(\geq\ell)} c_1(L,h^L)^n.
\end{equation}
Hence the last integral exists and we obtain
\begin{align*}
\sum_{j=\ell}^n(-1)^{j-\ell}\dim H^j(M,L^p\otimes E\otimes \cI(h^{L^p}))&
\leq\rank(E)\,\frac{(mp')^n}{n!}\int_{M(\geq\ell)} c_1(L,h^L)^n+o(p^n) \\
&=\rank(E)\,\frac{p^n}{n!}\int_{M(\geq\ell)} c_1(L,h^L)^n+o(p^n)
\end{align*}
The inequality \eqref{eq_hmi_qw} follows by summing up 
\eqref{eq_hmi_q} for $\ell$ and $\ell+1$. 
The proof of Theorem \ref{thm_hmi_q} is complete.
\end{proof}   

    
We conclude this section with the proof of Theorem \ref{thm_qcvx_bb}. 
We need the following result.
	
\begin{thm}[{\cite[Theorem 1.5]{Wh:21b}}]\label{thm_wk_est}
		Let $M$ be a $q$-convex manifold and $L,E$ be holomorphic line bundles on $M$. Let $h^L$ be a Hermitian metric on $L$ such that $c_1(L,h^L)\geq 0$ on a neighborhood $U$ of the exceptional set $K$ of $M$. Then there exists $C>0$ such that for every $j\geq q$ and $p\geq 1$,
		$\dim H^{j}(M,L^p\otimes E)\leq Cp^{n-j}$. 
\end{thm}

\begin{proof}[Proof of Theorem \ref{thm_qcvx_bb}]
As in the proof of Theorem \ref{thm_hmi_q}, 
let $\pi:\til M\rightarrow M$ be the proper modification provided 
by Theorem \ref{thm_tech}. 
So $\pi:\til M\setminus\til D\to M\setminus S(h^L)$ 
is biholomorphic and $\til M$ is  $q$-convex with exceptional set 
$\til D\cup\pi^{-1}(K)\subset\til U$, where $\til U=\pi^{-1}(U)$. 
By \eqref{e:c1}, it follows that 
$c_1(\oh L,h^{\oh L})=mc_1(\til L,h^{\til L})=
m\pi^*c_1(L,h^L)\geq0$ 
on $\til U\setminus\til D$. Since $h^{\oh L}$ is smooth, 
this implies that $c_1(\oh L,h^{\oh L})\geq0$ on $\til U$. 
Using \eqref{e:decomp}, we obtain that \eqref{e:cohomology} 
holds for $0\leq j\leq n$, where $p=mp'+m'$ and $F_{m'}$, 
$0\leq m'<m$, are the line bundles defined in \eqref{e:remainder}. 
Applying Theorem \ref{thm_wk_est} to $\wi M$, 
$(\widehat L,h^{\widehat{L}})$ and $F_{m'}$, 
we obtain for $p\geq m$ and $j\geq q$,
\[ \dim H^j(M, L^p\otimes E\otimes\cI(h^{L^p}))=
\dim H^j(\wi{M},\widehat{L}^{p'}\otimes F_{m'})
\leq C(p')^{n-j}\leq Cp^{n-j}, \]
which is the desired estimate.
\end{proof}

\subsection{Pseudoconvex domains}    

We establish here the singular holomorphic Morse inequalities 
on smooth pseudoconvex domains and weakly $1$-complete manifolds. 
The corresponding results fot smooth Hermitian metrics
were obtained in \cite{Bou:89,MM07,Mar92b}.

Let $M$ be a relatively 
compact domain with smooth boundary $bM$ in a complex manifold $V$. 
Let $\rho\in \cC^\infty(V,\R)$ be a defining function of $M$, 
i.e.\ $M=\{ x\in V:\,\rho(x)<0 \}$ and $d\rho\neq 0$ 
on the boundary $bM=\{x\in V: \rho(x)=0\}$. 
Let $T^{(1,0)}_xbM:=\{ v\in T_x^{(1,0)}V: \dbar\rho(v)=0 \}$
be the holomorphic tangent space of $bM$ at $x\in bM$. 
The Levi form $\cL_\rho$ is the restriction of $\dbar\ddbar\rho$
to the holomorphic tangent bundle $T^{(1,0)}bM$. 
The domain $M$ is called pseudoconvex if the Levi form 
$\cL_\rho$ is positive semidefinite. 
We need the following holomorphic Morse inequality 
for pseudoconvex domains.
\begin{thm}[{\cite[(3.5.25)]{MM07}}]\label{T:pcvxs}
Let $M\Subset V$ be a smooth pseudoconvex domain 
in a complex manifold $V$ and let $L,E$ be holomorphic 
vector bundles on $V$ with $\rank (L)=1$. Let $h^L$ be a 
Hermitian metric on $L$ such that $c_1(L,h^L)>0$ on $bM$. 
Then, as $p\rightarrow \infty$, we have
\[\dim H^0(M,L^p\otimes E)\geq\rank(E)\,\frac{p^n}{n!}
\int_{M(\leq 1)} c_1(L,h^L)^n+o(p^n).\]
\end{thm}	 

\begin{proof}[Proof of Theorem \ref{thm_pcvx}]
By shrinking $V$, we assume that $c_1(L, h^L)>0$ on $V\setminus M$.
Let $\rho$ be a defining function for $M$ and $\pi:\til V\rightarrow V$
be the proper modification from Theorem \ref{thm_tech}. 
Then  $\til D=\pi^{-1}(S(h^L))\subset\pi^{-1}(M)$ and 
$\pi:\til V\setminus\til D\to V\setminus S(h^L)$ is biholomorphic. 
So $\til M:=\pi^{-1}(M)$ is pseudoconvex with defining function 
$\rho\circ\pi$. Moreover, if $(\til L,h^{\til L})=(\pi^*L,\pi^*h^L)$ 
then $h^{\til L}$ is smooth and $c_1(\til L,h^{\til L})>0$ 
on a neighborhood of $b\til M$.
	
We write $p=mp'+m'$, $0\leq m'<m$. Let $\til E=\pi^*E$, 
and $\oh L,F_{m'}$ be the bundles defined in \eqref{e:dvd} 
and \eqref{e:remainder}, respectively. By \eqref{e:decomp} 
and \eqref{e:isoms} we have, for all $p$ sufficiently large, that
\[H^0(M,  L^p\otimes E\otimes \cI(h^{L^p}))\cong 
H^0(\til M,\oh L^{p'}\otimes F_{m'}).\]
Let $h^{\oh L}$ be the Hermitian metric on $\oh L$ provided 
by Theorem \ref{thm_tech} (B) and $\til M(\leq1)$ be the subset 
of $\til M$ where $R^{(\oh L,h^{\oh L})}$ is non-degenerate 
and has at most one negative eigenvalue. By \eqref{e:c1}, 
$c_1(\oh L,h^{\oh L})>0$ on $b\til M$, 
$\til M(\leq1)\setminus\til D=\pi^{-1}(M(\leq1))$ and 
\[\int_{\til M(\leq1)} c_1(\oh L,h^{\oh L})^n=
m^n\int_{\til M(\leq1)\setminus\til D} c_1(\til L,h^{\til L})^n=
m^n\int_{M(\leq1)} c_1(L,h^L)^n.\]	
In particular, the last integral exists.
Theorem \ref{thm_pcvx} now follows from 
Theorem \ref{T:pcvxs} applied to $(\oh L,h^{\oh L})$ and $F_{m'}$.
\end{proof}
Following Nakano \cite{Nak:70}, we call a manifold
\emph{weakly $1$-complete} if it admits a smooth plurisubharmonic 
exhaustion function $\varphi:M\to\R$. The holomorphic Morse inequalities
for weakly $1$-complete manifolds and line bundles with smooth
metrics appeared in \cite{Bou:89,MM07,Mar92b}. The general version
of \cite{Mar92b} for $q$-positive line bundles outside a compact set
answered a question of Ohsawa \cite[p.\ 218]{Oh:82}.
We state here a version of the holomorphic Morse inequalities
for positive line bundles outside a compact set.
Note that by the finiteness theorem due to Ohsawa
\cite[Ch.\ 3, Theorem 1.3]{Oh:82} if $L$, $E$
are bundles on a weakly $1$-complete manifold $M$ 
such that $\rank(L)=1$ and $L$ is positive 
outside a compact set there exists $p_0\in\N$
such that for every $p\geq p_0$ and for $j\geq1$ the spaces
$H^j(X,L^p\otimes E)$ are finite dimensional.
\begin{thm}\label{thm_hmi_q1}
Let $M$ be a weakly $1$-complete manifold
dimension $n$, and $L,E$ be holomorphic 
vector bundles on $M$ with $\rank(L)=1$. Let $h^L$ 
be a Hermitian metric on $L$ with algebraic singularities 
such that $S(h^L)$ is compact and $c_1(L,h^L)$ is positive
outside a compact set.
Then, for any $\ell\geq1$, 
the following strong and weak Morse inequalities hold as $p\to\infty$:
\begin{equation}
\sum_{j=\ell}^n(-1)^{j-\ell}
\dim H^j(M, L^p\otimes  E\otimes\cI(h^{ L^p}))
\leq \rank(E)\,\frac{p^n}{n!}
\int_{M(\geq\ell)}(-1)^\ell c_1(L,h^L)^n+o(p^n), \label{eq_hmi_q10}
\end{equation}
\begin{equation} 
\dim H^\ell(M, L^p\otimes  E\otimes \cI(h^{ L^p}))\leq 
\rank(E)\,\frac{p^n}{n!}\int_{M(\ell)}(-1)^\ell c_1(L,h^L)^n+o(p^n). 
\label{eq_hmi_q11}
\end{equation}
\begin{equation} 
\dim H^0(M, L^p\otimes  E\otimes \cI(h^{ L^p}))\geq 
\rank(E)\,\frac{p^n}{n!}\int_{M(\leq1)} c_1(L,h^L)^n+o(p^n). 
\label{eq_hmi_q12}
\end{equation}
\end{thm}
\begin{proof} 
If the metric $h^L$ is smooth the statement reduces to
\cite[Theorem 3.5.12]{MM07}, so the proof proceeds as above by
using \cite[Theorem 3.5.12]{MM07} on the blow-up.
\end{proof} 
\begin{rem}
In the same vein we can generalize \cite[Theorem, p.\ 897]{Mar92b}
for the case of a line bundle $L$ which is $q$-positive (that is, whose curvature
has $n-q+1$ positive eigenvalues) outside a compact set $K$.
In this case \eqref{eq_hmi_q10}, \eqref{eq_hmi_q11} hold 
for the cohomology groups 
$H^\ell(M_c, L^p\otimes  E\otimes \cI(h^{ L^p}))$,
$\ell\geq q$, on any sublevel set 
$M_c=\{\varphi<c\}$ containing $K$ and $S(h^L)$.
If we assume moreover that $X$ is endowed with a Hermitian
metric which is K\"ahler outside $K$ and $L$ is semi-positive outside $K$,
then the restriction morphism 
$H^\ell(M, L^p\otimes  E\otimes \cI(h^{ L^p}))\to
H^\ell(M_c, L^p\otimes  E\otimes \cI(h^{ L^p}))$ is an isomorphism
cf. \cite[Theorem 2.5, p.\ 221]{Oh:82}.
We deduce that the Morse inequalities \eqref{eq_hmi_q10}, \eqref{eq_hmi_q11}
hold in this case for 
$H^\ell(M, L^p\otimes  E\otimes \cI(h^{ L^p}))$, $\ell\geq q$.
\end{rem}

\subsection{Hyperconcave manifolds}
 A complex manifold $M$ is called \emph{hyperconcave} if 
 there exists a smooth function $\varphi:M\rightarrow (-\infty,u]$, 
 where $u\in\R$, such that $M_c:=\{\varphi>c\}\Subset M$
 for all $c\in (-\infty,u]$ and $\varphi$ is strictly plurisubharmonic 
 outside a compact subset (cf.\ \cite{Colt:90,Colt:91,MD06}). 
 A hyperconcave manifold is $1$-concave in the sense of 
 Andreotti-Grauert \cite{AG:62}, see Section \ref{S:qconc}.
 
If $X$ is a compact complex space with isolated singularities the 
regular locus $X_{\reg}$ is hyperconcave (see \cite[Example 3.4.2]{MM07}. 
A complete K\"{a}hler manifold of finite volume and bounded negative 
sectional curvature is hyperconcave 
(see e.g.\ \cite[Theorem 6.3.8]{MM07}). 
As in the compact case, Siegel's lemma holds for hyperconcave manifolds: 
\begin{equation}\label{e:Siegel}
\dim H^0(M, L^p\otimes K_M)\leq Cp^{\varrho_p},
\end{equation}
where $\varrho_p\leq \dim M$ is the maximal rank of 
the Kodaira map associated to $H^0(M, L^p\otimes K_M)$ 
(see \cite[Theorem 3.4.5, Remark 3.4.6]{MM07}).

We will need the following holomorphic Morse inequality for hyperconcave 
manifolds.
\begin{thm}\label{T:hcms}
Let $M$ be a hyperconcave manifold of dimension $n$ 
and $(L,h^L),(E,h^E)$ be Hermitian holomorphic line bundles 
on $M$ that are semi-positive outside a compact set. Then
\begin{equation}\label{e:hcms}
 \dim H^0_{(2)}(M,L^p\otimes E\otimes K_M)\geq 
\frac{p^n}{n!}\int_{M(\leq 1)}c_1(L,h^L)^n+o(p^n)\,,\,
\text{ as } p\rightarrow \infty,
\end{equation} 
where the set $M(\leq1)$ corresponds to the metric $h^L$ and 
$H^0_{(2)}(M,L^p\otimes E\otimes K_M)$ is the space of 
$L^2$-holomorphic sections of $L^p\otimes E\otimes K_M$ 
with respect to $h^L, h^E$ and any metric on $M$.
\end{thm}	
\begin{proof}
The case when $(E,h^E)$ is trivial was treated in 
\cite[Theorem 3.4.9]{MM07}.
In the general case,
we observe that \cite[Theorem 3.3.5 (i)]{MM07},
on which \cite[Theorem 3.4.9]{MM07} is based,
holds if we twist $L^p$ with a line bundle $E$
which is semi-positive outside a compact set, since the
crucial estimates \cite[(3.3.10-11)]{MM07} still hold
in this case. Thus, the proof of \cite[Theorem 3.4.9]{MM07}
goes through with only minor modifications.
 \end{proof}
 \begin{proof}[Proof of Theorem \ref{thm_hyper}]
 Let $\pi:\til M\rightarrow M$ be the proper modification provided 
 by Theorem \ref{thm_tech}, so 
 $\til D=\pi^{-1}(S(h^L))\subset\pi^{-1}(Z)$ and 
 $\pi:\til M\setminus\til D\to M\setminus S(h^L)$ 
 is biholomorphic. It is clear by the definition that $\til M$ 
 is hyperconcave. We write $p=mp'+m'$, $0\leq m'<m$, 
 and set
\[D_{m'}:=\sum_{j=1}^k \lfloor\, m'cc_j \rfloor\,\til D_j\,,\,\;E_{m'}
:=\til L^{m'}\otimes\cO_{\til M}(-D_{m'}).\]
We have by \eqref{e:decomp} that 
$\til L^p\otimes\cI(h^{\til L^p})=\oh L^{p'}\otimes E_{m'}$, 
where $\oh L$ is defined in \eqref{e:dvd}. 
Therefore, using \eqref{e:isom}, we obtain for $p$ sufficiently large that
\begin{equation}\label{e:cohom}
H^0(M,  L^p\otimes K_M\otimes \cI(h^{L^p}))\cong 
H^0(\til M,\til L^p\otimes K_{\til M}\otimes \cI(h^{\til L^p}))
=H^0(\til M,\oh L^{p'}\otimes E_{m'}\otimes K_{\til M}).
\end{equation}
Let $h^{\oh L}$ be the Hermitian metric of $\oh L$ from
Theorem \ref{thm_tech} (B) and $\til M(k)$ be the subset of 
$\til M$ where $R^{(\oh L,h^{\oh L})}$ is non-degenerate 
and has exactly $k$ negative eigenvalues. By \eqref{e:c1},
\[c_1(\oh L,h^{\oh L})=mc_1(\til L,h^{\til L})\geq0 \text{ on }
\til M\setminus\pi^{-1}(Z).\]
Applying \eqref{e:hcms} to $(\oh L,h^{\oh L})$ and using 
\eqref{e:Siegel}, we deduce that 
$0\leq\int_{\til M(0)} c_1(\oh L,h^{\oh L})^n<+\infty$.

Note that $E_{m'}$ carries a Hermitian metric $h_{m'}$ 
which is semi-positive outside a compact set. Indeed, let $s_{m'}$ 
be the canonical section of $\cO_{\til M}(-D_{m'})$ and $\eta_{m'}$ 
be the singular Hermitian metric on $\cO_{\til M}(-D_{m'})$ such 
that $|s_{m'}|_{\eta_{m'}}=1$ on $\til M\setminus D_{m'}$. 
On $\til M\setminus\til D$, the metric $h^{\til L^{m'}}\otimes\eta_{m'}$
is smooth and $c_1(\cO_{\til M}(-D_{m'}),\eta_{m'})=0$. 
So we can find a smooth metric $h_{m'}$ of $E_{m'}$ 
such that $h_{m'}=h^{\til L^{m'}}\otimes\eta_{m'}$ on 
$\til M\setminus K$ for some compact $K\supset\pi^{-1}(Z)$. 
Hence
\[c_1(E_{m'},h_{m'})=m'c_1(\til L,h^{\til L})\geq0 \text{ on }
\til M\setminus K.\]
Therefore we can apply Theorem \ref{T:hcms} to 
$(\oh L,h^{\oh L})$ and $(E_m',h_m')$. 
Using \eqref{e:c1} we infer that 
$\til M(\leq1)\setminus\til D=\pi^{-1}(M(\leq1))$ and 
\[\int_{\til M(\leq1)} c_1(\oh L,h^{\oh L})^n=
m^n\int_{\til M(\leq1)\setminus\til D} c_1(\til L,h^{\til L})^n=
m^n\int_{M(\leq1)} c_1(L,h^L)^n.\]	
In particular, $\int_{M(\leq 1)} c_1(L,h^L)^n\in\R$ exists. 
By \eqref{e:cohom} and \eqref{e:hcms} we obtain, as $p\to\infty$,
\begin{align*}
\dim H^0(M,  L^p\otimes K_M\otimes \cI(h^{L^p}))&\geq
\dim H^0_{(2)}(\til M,\oh L^{p'}\otimes E_{m'}\otimes K_{\til M}) \\
&\geq \frac{(p')^n}{n!}\int_{\til M(\leq 1)}c_1(\oh L,h^{\oh L})^n+o(p^n) \\
&= \frac{p^n}{n!}\int_{M(\leq 1)}c_1(L,h^L)^n+o(p^n).
\end{align*}
This is the desired estimate.
\end{proof}
Theorem \ref{thm_hyper} has the following immediate corollary. 
Recall that in analogy to the volume of a line bundle
\cite{Bouck:02}, the adjoint volume of a line bundle $L$ on 
a complex manifold $M$ of dimension $n$ is defined by
\[\vol^*(L):=\limsup_{p\rightarrow\infty}
\frac{n!}{p^n}\dim H^0(M,L^p\otimes K_M).\]
The volume of any line bundle on a hyperconcave 
manifolds is finite by \eqref{e:Siegel}.
 
\begin{cor}\label{cor_ajvol} 
Let $M$ be a hyperconcave manifold of dimension $n$ and $L$ be a holomorphic line bundle on $M$. 

\smallskip

(i) If $h^L$ is a Hermitian metric on $L$ with algebraic 
singularities such that $S(h^L)\subset Z$ and $c_1(L,h^L)\geq 0$ 
on $M\setminus Z$ for some compact $Z\subset M$, then
\[0\leq \int_{M(0)}c_1(L,h^L)^n\leq 
\vol^*(L)-\int_{M(1)}c_1(L,h^L)^n< \infty.\]
In particular, if $c_1(L,h^L)\geq 0$ on $M$, we have
$\int_{M}c_1(L,h^L)^n\leq\vol^*(L)<\infty$.
\smallskip

(ii) If $\varphi$ is a function on $M$ with algebraic 
singularities as in \eqref{e:algsing} such that $\varphi$ is smooth and plurisubharmonic 
on $M\setminus K$ for some compact $K\subset M$, then 
 \[0\leq\int_{M(0)}(i\dbar\ddbar\varphi)^n\leq 
 -\int_{M(1)}(i\dbar\ddbar\varphi)^n< +\infty.\]
\end{cor}  

\begin{proof}
$(i)$ This follows at once from Theorem \ref{thm_hyper}, 
since by \eqref{e:Siegel},
\[\dim H^0(M,  L^p\otimes K_M\otimes \cI(h^{L^p}))\leq
\dim H^0(M,  L^p\otimes K_M)<+\infty.\]
$(ii)$ We apply $(i)$ to the trivial bundle $L=M\times \C$ 
endowed with the singular Hermitian metric $|(x,1)|^2=e^{-2\varphi(x)}$, 
and we note that $\vol^*(L)=0$.      
\end{proof}
 
When $S(h^L)=\emptyset$, Corollary \ref{cor_ajvol} 
was obtained in \cite[Corollary 3.4.11]{MM07}. 
When $S(h^L)=\emptyset$ and $M$ is compact, $(ii)$ 
is motivated by the calculus inequalities derived from holomorphic 
Morse inequalities \cite{Siu:90}.
 
Another consequence of Theorem \ref{thm_hyper} is the following 
Siu-Demailly-Bonavero type criterion for Moishezon spaces with isolated 
singularities. Recall that a Moishezon space is a compact irreducible complex 
whose algebraic dimension is equal to its complex dimension \cite{Moi:66}. 

\begin{cor} \label{cor_moisp}
Let $X$ be a compact irreducible complex space of dimension $n\geq 2$ 
with at most isolated singularities and $L$ be a holomorphic line bundle 
on the regular locus $X_{\reg}$. Let $h^L$ be a Hermitian metric on $L$ 
with algebraic singularities such that $S(h^L)\subset Z$ and 
$c_1(L,h^L)\geq 0$ on $X_{\reg}\setminus Z$ for some compact set
$Z\subset X_{\reg}$\,. If 
\[\int_{X_{\reg}(\leq 1)}c_1(L,h^L)^n>0,\]
then $X$ is Moishezon.   	 
 \end{cor}     
 
 \begin{proof}
It is easy to see that $X_{\reg}$ is hyperconcave 
(see \cite[Example 3.4.2]{MM07}. 
By Theorem \ref{thm_hyper} we have 
\[\dim H^0(X_{\reg},L^p\otimes K_X)\geq 
\dim H^0(X_{\reg}, L^p\otimes K_X\otimes \cI(h^{L^p}))\geq Cp^n,\]
for some constant $C>0$ and all $p$ sufficiently large. 
It follows from Siegel's lemma \eqref{e:Siegel} that the Kodaira map 
associated to $H^0(X_{\reg}, L^p\otimes K_X)$ has maximal rank 
$\varrho_p=n$, for some $p$. Hence by \cite[Theorem 3.4.7]{MM07}, 
there exist $n$ algebraically independent meromorphic functions on 
$X_{\reg}$. These extend to meromorphic functions on $X$ 
by Levi's removable singularity theorem \cite[Theorem 3.4.8]{MM07}.
\end{proof} 
 
When $S(h^L)=\emptyset$, Corollary \ref{cor_moisp} 
was obtained in \cite[Theorem 3.4.10]{MM07}, \cite{Mar05}. When $X=X_{\reg}$, 
it reduces to Bonavero's criterion for Moishezon manifolds. 
Finally, when $X=X_{\reg}$ and $S(h^L)=\emptyset$,
this is the criterion of Siu-Demailly.   
 
\smallskip
 
We conclude this section by applying Theorem \ref{thm_hyper} to obtain a version of Bonavero's singular 
holomorphic Morse inequality for certain metrics with more 
general singularities than the algebraic ones. The setting is as follows.

Let $(X,\omega)$ be a compact Hermitian manifold 
of dimension $n$, $L$ be a holomorphic line bundle on $X$ 
and $h_0$ be a metric on $L$ with algebraic singularities. Then by \eqref{eq:wm2}, 
\[\dim H^0(X,L^p\otimes K_X\otimes \cI(h_0^{\otimes p}))\geq 
\frac{p^n}{n!}\int_{X(\leq 1)}c_1(L,h_0)^n+o(p^n),\, p\rightarrow\infty.\]

Set $A=S(h_0)$, for the singular locus of $h_0$. 
Let $U\subset X$ be an open set with 
$\overline U\subset X\setminus A$ 
and assume that there exists a psh function $\rho$ on $U$
such that $P=\{x\in U:\,\rho(x)=-\infty\}$ is compact and $\rho$
is smooth and strictly psh on $U\setminus P$. Moreover, assume that 
\begin{equation}\label{e:positive} 
c_1(L,h_0)\geq\varepsilon\omega\,\text{ on }U,
\end{equation}
for some constant $\varepsilon>0$. 
We fix a function $\chi\in C^\infty(X)$ such that 
$0\leq\chi\leq1$, $\supp\chi\subset U$ and $\chi=1$
on an open set $V\supset P$. For $t>0$ we define 
the singular metric $h_t$ on $L$ by
\begin{equation}\label{e:ht} 
h_t=h_0e^{-2t\chi\rho}.
\end{equation}
Note that the singular locus $S(h_t)=A\cup P$ and 
set $h_t^p:=h_t^{\otimes p}$.

\begin{thm}\label{T:hyperapp}
Let $(X,\omega)$ be a compact Hermitian manifold 
of dimension $n$, $L$ be a holomorphic line bundle on $X$ 
and $h_0$ be a metric on $L$ with algebraic singularities 
that verifies \eqref{e:positive}. Let $h_t$ be the singular Hermitian metric 
on $L$ defined in \eqref{e:ht}. Then there exists $t_0>0$ such 
that if $0<t\leq t_0$ we have, as $p\to\infty$,
\begin{equation}\label{e:hyperapp1}
\dim H^0(X,L^p\otimes K_X\otimes\cI(h_t^p))\geq 
\frac{p^n}{n!}\int_{X(\leq 1)}c_1(L,h_t)^n+o(p^n).
\end{equation}
If we assume in addition that 
\begin{equation}\label{e:Dem}
\int_{X(\leq1)}c_1(L,h_0)^n>0,
\end{equation}
then for every $\delta\in(0,1)$ there exists $t_1=t_1(\delta)>0$ such 
that if $0<t\leq t_1$ we have 
\begin{equation}\label{e:hyperapp2}
\dim H^0(X,L^p\otimes K_X\otimes\cI(h_t^p))\geq (1-\delta)\,
\frac{p^n}{n!}\int_{X(\leq 1)}c_1(L,h_0)^n+o(p^n),\;p\to\infty.
\end{equation}
\end{thm}

\begin{proof} Let $M=X\setminus P$. Then $M$ is hyperconcave, 
as the function $\chi\rho$ is smooth on $M$, strictly psh on $V\setminus P$, 
and $-\chi\rho$ is an exhaustion of $M$. Note that $h_t$ 
is a metric with algebraic singularities on $L|_M$. Indeed, $h_t=h_0$ 
on $X\setminus U$ and $h_t$ is smooth on $U\setminus P$. 
Since $\chi\rho=\rho$ is psh on $V$, it follows using \eqref{e:positive} that 
\begin{equation}\label{e:posU}
c_1(L,h_t)=c_1(L,h_0)+t\,\ddc(\chi\rho)\geq\frac{\varepsilon}{2}\,\omega
\end{equation}
holds on $U$, for $0<t\leq t_0$ and some $t_0>0$.
The set $Z=X\setminus U$ is compact and contained in $M$, 
$S(h_t)\cap M=A\subset Z$, and $c_1(L,h_t)\geq0$ on $M\setminus Z$. 
Hence by Theorem \ref{thm_hyper},
\begin{equation}\label{e:hyperconc}
\dim H^0(M, L^p|_M\otimes K_M\otimes\cI(h_t^p))\geq 
\frac{p^n}{n!}\int_{X(\leq 1)}c_1(L,h_t)^n+o(p^n), \text{ as }p\to\infty.
\end{equation}

Let $S\in H^0(M,L^p|_M\otimes K_M\otimes\cI(h_t^p))$ 
and $x\in P$. Fix a coordinate neighborhood $W$ of $x$ 
such that $\chi=1$ and $L$ has a holomorphic frame $e_W$ on $W$. 
Then $S=f\,e_W^{\otimes p}\otimes(dz_1\wedge\ldots\wedge dz_n)$, 
where $f\in\cO_X(W\setminus P)$ verifies
$\int_W|f|^2e^{-2tp\rho}\,d\lambda<+\infty$ with $\lambda$ 
the Lebesgue measure on $W$. Since $\rho$ is upper bounded 
this implies that $\int_W|f|^2\,d\lambda<+\infty$. 
Hence $f$ extends to a holomorphic function on $W$, 
since pluripolar sets are removable for square integrable 
holomorphic functions (see e.g.\ \cite[Theorem 5.17]{Oh:02}). 
We conclude that $S$ extends to a section of 
$L^p\otimes K_X\otimes\cI(h_t^p)$, so 
\[H^0(M,L^p|_M\otimes K_M\otimes \cI(h_t^p))=
H^0(X,L^p\otimes K_X\otimes\cI(h_t^p)).\]
Thus \eqref{e:hyperapp1} follows from \eqref{e:hyperconc}.

Assume next that \eqref{e:Dem} holds and let $\delta\in(0,1)$. 
We claim that there exists $t_1=t_1(\delta)>0$ such that 
\begin{equation}\label{e:comp}
\int_{X(\leq 1,\,h_t)}c_1(L,h_t)^n\geq(1-\delta)\int_{X(\leq 1,\,h_0)}c_1(L,h_0)^n, \text{ for } 0<t\leq t_1.
\end{equation}
Indeed, $U\subset X(0,h_0)$ by \eqref{e:positive}, 
and $U\setminus P\subset X(0,h_t)$ by \eqref{e:posU}, 
provided that $t\leq t_0$. Since $h_t=h_0$ on $X\setminus U$, 
it follows that $X(1,h_t)=X(1,h_0)\subset X\setminus U$ and 
$X(0,h_t)=X(0,h_0)\setminus P$. Therefore
\begin{align*}
\int_{X(\leq 1,\,h_t)}c_1(L,h_t)^n &=
\int_{X(1,\,h_0)}c_1(L,h_0)^n+
\int_{X(0,\,h_0)\setminus P}c_1(L,h_t)^n\\
&=\int_{X(\leq1,\,h_0)}c_1(L,h_0)^n+
\int_{X(0,\,h_0)\setminus P}c_1(L,h_t)^n-c_1(L,h_0)^n\\
&=\int_{X(\leq1,\,h_0)}c_1(L,h_0)^n+
\int_{U\setminus P}c_1(L,h_t)^n-c_1(L,h_0)^n.
\end{align*}
Since $c_1(L,h_t)=c_1(L,h_0)+t\,\ddc\rho>c_1(L,h_0)>0$
on $V\setminus P$ we infer that
\[\int_{V\setminus P}c_1(L,h_t)^n-c_1(L,h_0)^n>0.\]
We conclude by above that 
\[\int_{X(\leq 1,\,h_t)}c_1(L,h_t)^n>
\int_{X(\leq1,\,h_0)}c_1(L,h_0)^n+
\int_{U\setminus V}c_1(L,h_t)^n-c_1(L,h_0)^n.\]
Now \eqref{e:comp} follows using \eqref{e:Dem}
and the fact that 
\[c_1(L,h_t)^n-c_1(L,h_0)^n=t\,
\ddc(\chi\rho)\wedge\sum_{j=0}^{n-1}\big(c_1(L,h_0)+
t\,\ddc(\chi\rho)\big)^j\wedge 
c_1(L,h_0)^{n-1-j}=O(t)\]
on the compact set $\overline U\setminus V$. 
Since \eqref{e:hyperapp1} and \eqref{e:comp} imply 
\eqref{e:hyperapp2}, the proof is complete. 
\end{proof}

\begin{example}\label{R:hyper-ex}
We observe that there are many situations in which 
Theorem \ref{T:hyperapp} applies, if one chooses $U$ 
to be an open coordinate set (or a disjoint union of such sets) 
and $\rho$ a psh function on $U$ with the desired properties. 
Simple examples can be given as follows. Let $D$ be a polydisc 
centered at $0$ in ${\mathbb C}^n$ and write 
$z=(z',z_n)\in{\mathbb C}^{n-1}\times\mathbb C$. 
Let  $P_1=\{\zeta_j:\,j\geq1\}$ 
be an at most countable compact subset of $D$ and $\varepsilon_j>0$
be chosen small enough so that the function 
\[u_1(z)=\sum_{j=1}^\infty\varepsilon_j\log\|z-\zeta_j\|\]
is psh on $D$ and smooth on $D\setminus P$. 
Let next $P_2\subset D\cap\{z'=0\}$ be a compact polar set. 
By a theorem of Evans, there exists a probability measure 
$\mu$ supported on $P_2$ whose logarithmic potential 
\[L_\mu(z_n)=\int_{\mathbb C}\log|z_n-w|\,d\mu(w)\]
is subharmonic on $\mathbb C$, harmonic on $\mathbb C\setminus P_2$ 
and equal to $-\infty$ on $E$. Set
 \[u_2(z)={\max}_\eta\{\log\|z'\|,L_\mu(z_n)\},\] 
where $0<\eta<1$ is fixed and ${\max}_\eta$ 
denotes the regularized maximum as constructed in 
\cite[Lemma 5.18]{Dem_agbook}. 
Then the functions $\rho_j(z)=u_j(z)+\|z\|^2$, $z\in D$, $j=1,2$, 
are strictly psh on $D$, smooth on $D\setminus P_j$, and equal to 
$-\infty$ on $P_j$.
\end{example}

 \subsection{$q$-concave manifolds}\label{S:qconc}
According to \cite{AG:62}, a complex manifold $M$ of dimension $n$ 
is called \emph{$q$-concave} for some $q\in\{1,\ldots,n\}$, 
if there exists a smooth function 
$\varphi:M\longrightarrow (a,b]$, where $\,a\in\R\cup\{-\infty\}$ and $b\in\R$, 
 so that
 $M_c:=\{\varphi>c\}\Subset M$ for all $c\in(a,b]$ and there
 exists a compact subset $K\subset M$ such that
 $i\partial\ddbar\varphi$ has at least $n-q+1$ positive eigenvalues on 
 $M\setminus K$. 
By the Andreotti-Grauert theory \cite{AG:62}, $H^j(X,\cF)$ is finite dimensional
for any $j\leq n-q-1$ and any coherent analytic sheaf $\cF$ on a $q$-concave 
manifold $X$.

 By applying \cite[Corollary 4.3]{Mar96} and the same method as above, we obtain:
 \begin{thm}
 Let $M$ be a $q$-concave manifold of 
dimension $n\geq 3$, and $L,E$ be holomorphic vector bundles 
on $M$ with $\rank(L)=1$. Let $h^L$ be a Hermitian metric on 
$L$ with algebraic singularities such that $S(h^L)\subset Z$ 
and $c_1(L,h^L)\leq 0$ on $X\setminus Z$ for some compact $Z\subset M$. 
Then, for any $\ell\leq n-q-2$, we have as $p\rightarrow \infty$, 
\begin{align*}  
\sum_{j=0}^\ell(-1)^{\ell-j}\dim H^j(M,L^p\otimes E\otimes \cI(h^{L^p}))&
\leq \rank(E)\,\frac{p^n}{n!}\int_{M(\leq\ell)}(-1)^\ell c_1(L,h^L)^n+o(p^n),\\
\dim H^\ell(M,L^p\otimes E\otimes\cI(h^{L^p}))&\leq 
\rank(E)\,\frac{p^n}{n!}\int_{M(\ell)}(-1)^\ell c_1(L,h^L)^n+o(p^n).
\end{align*}
\end{thm}    

\noindent 
\textbf{Acknowledgements.}
Huan Wang thanks Claudio Arezzo, Xiaoshan Li, Bo Liu and Guokuan Shao  
for support during preparing this manuscript.


\def\cprime{$'$}

\end{document}